\documentclass[12pt]{amsart}
\usepackage{ifthen,verbatim}
\usepackage{graphicx}
\usepackage{mathrsfs}
\usepackage{amsmath,amssymb,amsthm}
\usepackage{tikz-cd}
\usepackage{color}
\usepackage{caption}
\usepackage{subcaption}
\numberwithin{equation}{section}
\nonstopmode
\usepackage{amsfonts}
\usepackage{amssymb}
\usepackage[english]{babel}
\usepackage[utf8x]{inputenc}
\usepackage[autostyle]{csquotes}

\usepackage{float}
\usepackage{tabularx}
\newcolumntype{C}{>{\centering\arraybackslash}X}

\usepackage{tikz-cd}
\usepackage{color}
\usepackage{caption}

\newtheorem{theorem}{Theorem}
\newtheorem{lemma}{Lemma}

\usepackage{url}
\usepackage{graphicx}
\usepackage{multirow}
\usepackage{tabularx}
\usepackage{longtable}
\usepackage{enumerate}

\usepackage{tikz-cd}

\usepackage{bm}
\usepackage{amsmath,mathtools}
\usepackage{amsfonts}

\newcommand{\RomanNumeralCaps}[1]
    {\MakeUppercase{\romannumeral #1}}

\renewcommand{\Im}{{\operatorname{Im}\,}}
\usepackage[skip=1pt,font=scriptsize]{caption}

\usepackage{dirtytalk}

\usepackage{mathtools}

\begin{document}
\title
{Automorphic Forms and Holomorphic Functions on the Upper Half-plane}

\makeatletter\def\thefootnote{\@arabic\c@footnote}\makeatother

\author[Md. S. Alam]{Md. Shafiul Alam}
\address{
Department of Mathematics, University of Barishal, Barishal-8254, Bangladesh}
\email{msalam@bu.ac.bd, shafiulmt@gmail.com}

\begin{abstract}
We define a set of holomorphic functions in terms of the Hauptmodul of a quotient Riemann surface and prove that these functions are holomorphic on the upper half-plane. It is also shown that these functions are automorphic forms of weight $k$ with respect to a Fuchsian group.
\end{abstract}

\maketitle

\section{Introduction}

The group $\operatorname{SL}(2,\mathbb{R})$ is defined by
$$ 
\operatorname{SL}(2,\mathbb{R})=\Bigg\{\displaystyle\begin{pmatrix} a & b\\ c & d \end{pmatrix}: a,b,c,d\in\mathbb{R},~ad-bc=1\Bigg\}
$$
and the group $\operatorname{PSL}(2,\mathbb{R})=\operatorname{SL}(2, \mathbb{R})/\{\pm I_2\}$, where $I_2$ is the $2\times2$ identity matrix (see \cite[Chapter \RomanNumeralCaps{7}]{serre}). Let $\mathbb{H}$ denote the upper half-plane $\{\tau\in\mathbb{C}: \Im \tau>0\}$. The boundary of $\mathbb{H}$ is $\mathbb{R}\cup\infty$. The group $\operatorname{PSL}(2,\mathbb{R})$ acts on $\mathbb{H}$ as follows:
   \[\tau\mapsto \gamma\cdot\tau= \frac{a\tau+b}{c\tau+d},\, \text{for}\,\gamma=\begin{pmatrix}
   a & b\\ c & d
   \end{pmatrix}\in \operatorname{PSL}(2,\mathbb{R}), \, \tau\in\mathbb{H}.\]
All transformations of $\operatorname{PSL}(2,\mathbb{R})$ are conformal. \vspace{2mm}

A Fuchsian group is a discrete subgroup of $\operatorname{PSL}(2,\mathbb{R})$, i.e., it is a group of orientation-preserving isometries of $\mathbb{H}$. Study of Fuchsian group is a very interesting topic in many fields of Mathematics. Many mathematicians studied Fuchsian group and various subgroups of Fuchsian group, for example, see \cite{gabai}, \cite{patterson}, \cite{singer} and \cite{singer2}. The Hecke group which is a subgroup of Fuchsian group is studied in \cite{alam2} and \cite{alam1} to investigate Ramanujan's modular equations. \vspace{2mm}

Let $\gamma=\begin{pmatrix} a & b\\ c & d \end{pmatrix}\in \operatorname{PSL}(2, \mathbb{R})$ and let $\operatorname{tr}(\gamma)$ denote the trace of $\gamma$, then the element $\gamma$ is said to be elliptic, parabolic and hyperbolic when $|\operatorname{tr}(\gamma)|<2,\, |\operatorname{tr}(\gamma)|=2$ and $|\operatorname{tr}(\gamma)|>2$, respectively. If $\Gamma\subset \operatorname{PSL}(2,\mathbb{R})$ is a Fuchsian group and $\gamma\in \Gamma$ is an elliptic element, then a point $\tau\in \mathbb{H}$ is called an elliptic point of $\Gamma$ if $\gamma(\tau)=\tau$. Also, for a parabolic element $\sigma\in\Gamma$, a point $x\in\mathbb{R}\cup\{\infty\}$ is called a cusp of $\Gamma$ if $\sigma(x)=x$. If a Fuchsian group $\Gamma$ acts on $\mathbb{H}$ properly discontinuously, then we have the quotient Riemann surface $\Gamma\backslash\mathbb{H}$.  For a detailed discussion, see \cite{beardon} and \cite{Katok:fg}. \vspace{2mm}

Let $\mathbb{H}^*$ denote the union of the upper half-plane $\mathbb{H}$ and the set of cusps of a Fuchsian group $\Gamma$. Suppose $\begin{pmatrix}
    a & b\\
    c  & d
\end{pmatrix}\in \Gamma$, $\tau\in \mathbb{H}$ and $f:\mathbb{H}\rightarrow{\mathbb{C}}$ is a holomorphic function. Then the function $f$ is called an automorphic form of weight $k$ with respect to $\Gamma$ if
\begin{equation*}
    f\bigg(\frac{a\tau+b}{c\tau+d}\bigg)=(c\tau+d)^k f(\tau).
\end{equation*}
If $k=0$, then  
\begin{equation*}
    f\bigg(\frac{a\tau+b}{c\tau+d}\bigg)=f(\tau)
\end{equation*}
and $f$ is called an automorphic function. When the genus of the quotient Riemann surface $\Gamma\backslash\mathbb{H}^*$ is zero, an automorphic function is called a Hauptmodul. If an automorphic function has no poles, then it is constant according to the consequence of maximum modulus principle. For details, we refer the reader to \cite{bump}, \cite{diamond}, \cite{iwa}, and \cite{miyake}. \vspace{2mm}

Let $F$ be the fundamental domain for the Fuchsian group $\Gamma$. Let $X$ and $\hat X$ denote the quotient Riemann surfaces $\Gamma\backslash\mathbb{H}$ and $\Gamma\backslash\mathbb{H}^*$, respectively. If $F$ is compact, then it has finitely many vertices which are elliptic points and cusps of  $\hat X=\Gamma\backslash\mathbb{H}^*$. Let $P_1, \ldots, P_r$ be the vertices whose orders are $n_1, n_2,\ldots,n_r$, respectively. If the number of elliptic elements and cusps of $\Gamma$ are $m$ and $l$, respectively, then $m+l=r$. If $g$ is the genus of $\hat X$, then we say that $\Gamma$ has signature $(g;n_1,\ldots,n_n)$. For more detailed discussion, reader may consult Section 2.1 of \cite{bers}, Chapter 4 of \cite{Katok:fg}, and Section 2 of \cite{singer}. Let us denote by $A_k$ the space of automorphic forms of weight $k$ with respect to $\Gamma$. The basis for $A_k$ on a Shimura curve $X$ with genus $0$ is determined in Theorem 4 of \cite{Yifan}. The following theorem is written according to Theorem 2.23 of \cite{Shimura} to determine the dimension of $A_k$.

\begin{theorem}[$\text{\cite[Theorem~2.23]{Shimura}}$]\label{th:shimura2.23}

For a Fuchsian group $\Gamma$ with signature $(g;n_1,\ldots,n_r)$, let $g$ be the genus of the compact quotient Riemann surface $\hat X=\Gamma\backslash\mathbb{H}^*$. Then, the dimension, $\operatorname{dim} A_k$, of $A_k$ for an even integer $k$ is given by
\begin{align*}
    \operatorname{dim} A_k=
    \begin{cases}
      0 &  \text{if $k<0$,}\\
      1 & \text{if $k=0$,}\\
      g & \text{if $k=2$,}\\
      (g-1)(k-1)+\displaystyle\sum_{i=1}^r\bigg\lfloor \frac{k}{2}\big(1-\frac{1}{n_i}\big)\bigg\rfloor & \text{if $k\geq4$}.
    \end{cases}
\end{align*}
\end{theorem}

In the following section, we present our main results and their proofs.

\section{Main Results}

Let $(0;n_1,\ldots,n_r)$ be the signature of a Fuchsian group $\Gamma$, i.e., the genus of the quotient Riemann surface $\hat X=\Gamma\backslash\mathbb{H}^*$ is $0$ and let $d=\operatorname{dim} A_k$. Then, for an even integer $k\geq4$, we have from Theorem 1 $$d=1-k+\displaystyle\sum_{i=1}^r\bigg\lfloor \frac{k}{2}\big(1-\frac{1}{n_i}\big)\bigg\rfloor.$$ 

In the following theorem, we define the functions $h_j$ for $j=0,\ldots, d-1$ so that the functions are holomorphic on $\mathbb{H}$. Also, these functions are automorphic forms of weight $k$ with respect to $\Gamma$.

\begin{theorem}\label{th:yifanth4}
Consider the Fuchsian group $\Gamma$ with signature $(0; n_1, \ldots, n_r)$ and the compact quotient Riemann surface $\hat X=\Gamma\backslash\mathbb{H}^*$. Let $\tau_1,\ldots,\tau_r$ be the inequivalent vertices (elliptic points or cusps of $\hat X$) of the fundamental domain of $\Gamma$ of orders $n_1, \ldots, n_r$, respectively, and let $w(\tau)$ be a Hauptmodul of $\hat X$. For an even integer $k\geq4$, let 
\begin{align*}
a_i=\Big\lfloor \frac{k}{2}\Big(1-\frac{1}{n_i}\Big)\Big\rfloor \quad\text{and}\quad
    d=\operatorname{dim} A_k= 1-k+\sum_{i=1}^r a_i.
\end{align*}
If $w(\tau_i)=w_i$ for $i=1,\ldots,r$ and the functions  
\begin{equation}\label{g1}
    h_j(\tau)=\frac{\big(w'(\tau)\big)^{k/2} \big(w(\tau)\big)^j}{\displaystyle\prod_{i=1,w_i\neq\infty}^r \big(w(\tau)-w_i\big)^{a_i}}
\end{equation}
for $j=0,\ldots, d-1$ and $\tau\in \mathbb{H}$, then $h_j(\tau)$ is holomorphic on $\mathbb{H}$.
\end{theorem}

\begin{proof}
We need to consider the following three cases:\vspace{2mm}
\begin{enumerate}[(i)]
    \item the Hauptmodul $w(\tau)$ does not have any pole at the points $\tau_i$ for $i=1,\ldots,r$;\vspace{2mm}
    \item the Hauptmodul $w(\tau)$ has a pole at one of the points $\tau_i$ for $i=1,\ldots,r$; \vspace{2mm}
    \item the Hauptmodul $w(\tau)$ has a pole at another point, say $\tau=\tau_0$, except the points $\tau_1,\ldots,\tau_r$.
\end{enumerate}\vspace{2mm}
If a function has a zero of order $\geq 0$ and has a pole of order $\leq 0$ at a point, then there is no principal part in the expansion of the function at that point, i.e., the function is holomorphic. Thus, we have to show that $h_j$ has a zero of order $\geq0$ at $\tau=\tau_i$ for Case (i), $h_j$ has a pole of order $\leq0$ at $\tau=\tau_i$ for Case (ii) and $h_j$ has a pole of order $\leq0$ at $\tau=\tau_0$ for Case (iii). \vspace{2mm}  

Case (i): If $w(\tau)$ does not have any pole at $\tau_i$, then $w(\tau_i)=w_i\neq \infty$ for $i=1,\ldots,r$. Since $\tau_i$ is a vertex of order $n_i$, in a neighbourhood of $\tau=\tau_i$, we have
\begin{align}
    w(\tau)-w(\tau_i)=b_i (\tau - \tau_i)^{n_i}+O\big((\tau - \tau_i)^{n_i+1}\big)
   \end{align}
   or,
   \begin{align}
    w(\tau)-w_i= (\tau - \tau_i)^{n_i} w^*(\tau),
   \end{align}
where $b_i\in \mathbb{C}\setminus \{0\}$, $w^*(\tau)$ is analytic in a neighbourhood of $\tau=\tau_i$ and $w^*(\tau_i)\neq 0$ for $i=1,\ldots,r$. Therefore, in a neighbourhood of $\tau=\tau_i$, one can define a single-valued analytic $n_i$-th root of $(w-w_i)$ and this can be done at all points which are equivalent to $\tau_i$ under the action of the Fuchsian group $\Gamma$. Since $w(\tau)-w_i\neq 0$ for $\tau\neq\tau_i$ and $(w-w_i)$ is analytic on the other part of $\mathbb{H}$, its $n_i$-th root is analytic at each point of the remainder of $\mathbb{H}$. As $(w(\tau)-w_i)^{n_i}$ is locally analytic and single-valued at each $\tau\in\mathbb{H}$, so it follows from monodromy theorem that a single-valued and analytic $n_i$-th root of $(w-w_i)$ can be defined on the whole $\mathbb{H}$. 

From (2.3), we observe that $\big(w(\tau)-w_i\big)$ has a zero of order $n_i$ at $\tau=\tau_i$ and $\displaystyle \prod_{i=1,w_i\neq\infty}^r \big(w(\tau)-w_i\big)^{a_i}$ has a zero of order $n_ia_i=\displaystyle n_i\Big\lfloor\frac{k}{2}\Big(1-\frac{1}{n_i}\Big)\Big\rfloor$ at $\tau=\tau_i$. Also, we have from (2.2)
\begin{align}
    w'(\tau)=b_i n_i (\tau - \tau_i)^{n_i-1}+O\big((\tau - \tau_i)^{n_i}\big).
   \end{align}
Consequently, at $\tau=\tau_i$, $\big(w'(\tau)\big)^{k/2}$ has a zero of order $\displaystyle \frac{k}{2}(n_i-1)$. Since $$\displaystyle \frac{k}{2}(n_i-1)-\displaystyle n_i\Big\lfloor\frac{k}{2}\Big(1-\frac{1}{n_i}\Big)\Big\rfloor\geq 0,$$ we conclude from (\ref{g1}) that $h_j$ has a zero of order $\geq0$ at $\tau=\tau_i$. Hence $h_j$ is holomorphic on $\mathbb{H}$.\vspace{2mm}

Case (ii): Assume that $w(\tau)$ has a pole at one of the points $\tau_i$ for $i=1,\ldots,r$. Without loss of generality, suppose $w(\tau)$ has a pole at $\tau_1$, i.e., $w(\tau_1)=w_1=\infty$. Since $\tau_1$ is a vertex of order $n_1$, it follows that
\begin{align*}
    w(\tau)=\frac{b_1}{(\tau - \tau_1)^{n_1}}+O\big((\tau - \tau_1)^{1-n_1}\big),\quad b_1\in \mathbb{C}\setminus \{0\}
\end{align*}\\
and
\begin{align*}
    w'(\tau)=-\frac{b_1n_1}{(\tau - \tau_1)^{n_1+1}}+O\big((\tau - \tau_1)^{-n_1}\big).
\end{align*}
In this case, from (\ref{g1}) we have
\begin{equation}\label{g2}
    h_j(\tau)=\frac{\big(w'(\tau)\big)^{k/2} \big(w(\tau)\big)^j}{\displaystyle\prod_{i=2,w_i\neq\infty}^r \big(w(\tau)-w_i\big)^{a_i}}.
\end{equation}
Now, suppose that $h_j(\tau)$ defined in (\ref{g2}) has a pole of order $N$ at $\tau=\tau_1$. Since $w(\tau)$ has a pole of order $n_1$ at $\tau=\tau_1$, $\big(w'(\tau)\big)^{k/2}$ has a pole of order $\displaystyle\frac{k}{2}(n_1+1)$ and $\displaystyle\prod_{i=2,w_i\neq\infty}^r \big(w(\tau)-w_i\big)^{a_i}$ has a pole of order $\displaystyle n_1\sum_{i=2}^ra_i$ at $\tau=\tau_1$. As $j$ varies from $0$ to $d-1$, so the maximum value of $j$ is $d-1=\displaystyle\sum_{i=1}^r a_i-k$. Hence $\big(w(\tau)\big)^j$ has a pole of order at most $n_1\Big(\displaystyle\sum_{i=1}^r a_i-k\Big)$ at $\tau=\tau_1$. Therefore, we have
\begin{align*}
    N&\leq \frac{k}{2}(n_1+1)+n_1\Big(\displaystyle\sum_{i=1}^r a_i-k\Big)-n_1\sum_{i=2}^ra_i\\ &=\frac{k}{2}(n_1+1)+n_1\Big(\sum_{i=1}^r\Big\lfloor \frac{k}{2}\Big(1-\frac{1}{n_i}\Big)\Big\rfloor-k\Big)-n_1\sum_{i=2}^r{\Big\lfloor \frac{k}{2}\Big(1-\frac{1}{n_i}\Big)\Big\rfloor}\\
    &=-\frac{k}{2}\Big(n_1-1\Big)+n_1\Big\lfloor \frac{k}{2}\Big(1-\frac{1}{n_1}\Big)\Big\rfloor \leq 0.
\end{align*}
Since $N\leq0$, it follows that there is no principal part in the expansion of $h_j$, i.e., $h_j$ is holomorphic on $\mathbb{H}$.\vspace{2mm}

Case (iii): Suppose that $w(\tau)$ has the value $\infty$ at the point $\tau=\tau_0$ and $w(\tau_i)\neq \infty$ for $i=1,\ldots, r$. Therefore, $w(\tau)$ has a simple pole at $\tau_0$ and we have
\begin{align}\label{c3w}
    w(\tau)=\frac{b_0}{(\tau - \tau_0)}+O(1),\quad b_0\in \mathbb{C}\setminus \{0\}
\end{align} and
\begin{align}\label{c3w'}
    w'(\tau)=-\frac{b_0}{(\tau - \tau_0)^{2}}+O(1).
\end{align}
Let $N_0$ be the order of the pole of $h_j$ defined in (\ref{g1}) at $\tau=\tau_0$. From (\ref{c3w}) and (\ref{c3w'}), we observe that $\big(w'(\tau)\big)^{k/2}$ has a pole of order $k$, $\displaystyle\prod_{i=1,w_i\neq\infty}^r \big(w(\tau)-w_i\big)^{a_i}$ has a pole of order $\displaystyle \sum_{i=1}^ra_i$ and $w^j$ has a pole of order at most $d-1=\displaystyle \sum_{i=1}^ra_i-k$. Therefore, from (\ref{g1}), it follows that
\begin{align*}
    N_0\leq k+\sum_{i=1}^ra_i-k- \sum_{i=1}^ra_i=0,
\end{align*}
which implies that $h_j$ is holomorphic on $\mathbb{H}$ in this case also. 
\end{proof}

\begin{lemma}
The functions $h_j$ for $j=0,\ldots, d-1$ defined in (\ref{g1}) is an automorphic form of weight $k$ with respect to the Fuchsian group $\Gamma$.
\end{lemma}

\begin{proof}

For $j=0,\ldots, d-1$ and $a_i=\Big\lfloor \frac{k}{2}\Big(1-\frac{1}{n_i}\Big)\Big\rfloor$, we have to show that  $$h_j\Big(\frac{a\tau+b}{c\tau+d}\Big)=(c\tau+d)^k h_j(\tau),$$
where $\begin{pmatrix}
    a & b\\
    c  & d
\end{pmatrix}\in \Gamma$ and $\tau\in \mathbb{H}$. Since $w(\tau)$ is a Hauptmodul of $\hat X$, i.e., $w(\tau)$ is an automorphic function, thus we have $$w\Big(\frac{a\tau+b}{c\tau+d}\Big)=w(\tau)$$ and $$w'\Big(\frac{a\tau+b}{c\tau+d}\Big)=(c\tau+d)^{1/2}w'(\tau).$$
Now,
\begin{align*}
    h_j\Big(\frac{a\tau+b}{c\tau+d}\Big)&=\frac{\displaystyle\bigg(w'\Big(\frac{a\tau+b}{c\tau+d}\Big)\bigg)^{k/2} \bigg(w\Big(\frac{a\tau+b}{c\tau+d}\Big)\bigg)^j}{\displaystyle\prod_{i=1,w_i\neq\infty}^r \bigg(w\Big(\frac{a\tau+b}{c\tau+d}\Big)-w_i\bigg)^{a_i}}\\
    &=\frac{(c\tau+d)^k\big(w'(\tau)\big)^{k/2} \big(w(\tau)\big)^j}{\displaystyle\prod_{i=1,w_i\neq\infty}^r \big(w(\tau)-w_i\big)^{a_i}}\\
    &=(c\tau+d)^k h_j(\tau).
\end{align*}
Thus, $h_j$ is an automorphic form of weight $k$ with respect to $\Gamma$.   
\end{proof}


\def\cprime{$'$} \def\cprime{$'$} \def\cprime{$'$}
\providecommand{\bysame}{\leavevmode\hbox to3em{\hrulefill}\thinspace}
\providecommand{\MR}{\relax\ifhmode\unskip\space\fi MR }

\providecommand{\MRhref}[2]{%
  \href{http://www.ams.org/mathscinet-getitem?mr=#1}{#2}
}
\providecommand{\href}[2]{#2}

\end{document}